\documentclass[11pt]{amsart}

\usepackage[margin=1in]{geometry}

\usepackage{amsmath, amscd, amsthm, amssymb,hyperref,comment}

\def\C{{\mathbf C}}
\def\R{{\mathbf R}}
\def\Z{{\mathbf Z}}
\def\Q{{\mathbf Q}}
\def\A{{\mathbf A}}
\def\g{{\mathfrak g}}
\def\m{{\mathfrak m}}
\def\n{{\mathfrak n}}

\def\sl{{\mathfrak sl}}

\def\O{{\mathbb O}}

\def\H{{\mathcal{H}}}

\newtheorem{theorem}{Theorem}[section]

\newtheorem{lemma}[theorem]{Lemma}
\newtheorem{proposition}[theorem]{Proposition}
\newtheorem{corollary}[theorem]{Corollary}
\newtheorem{claim}[theorem]{Claim}

\theoremstyle{definition}
\newtheorem{definition}[theorem]{Definition}

\theoremstyle{remark}
\newtheorem{remark}[theorem]{Remark}

\newcommand{\mm}[4]{\left(\begin{smallmatrix} #1 & #2\\ #3 & #4\end{smallmatrix}\right)}

\DeclareMathOperator{\tr}{tr}

\DeclareMathOperator{\SO}{SO}

\DeclareMathOperator{\Sp}{Sp}

\DeclareMathOperator{\GL}{GL}

\DeclareMathOperator{\End}{End}

\begin{document}
	\title{Automatic convergence for holomorphic modular forms}
	\author{Aaron Pollack}
	\address{Department of Mathematics\\ University of California San Diego\\ La Jolla, CA USA}
	\email{apollack@ucsd.edu}
	\thanks{Funding information: AP has been supported by the NSF via grant number 2144021.}
	
	\begin{abstract} We prove an automatic convergence theorem for holomorphic modular forms on tube domains.  The argument works in some generality, and covers in particular the case of orthogonal groups, symplectic groups, unitary and quaternion unitary groups, and the exceptional group $E_7$.  
	\end{abstract}
	
	\maketitle
	
	\setcounter{tocdepth}{1}
	\tableofcontents
\section{Introduction} 
The purpose of this paper is to prove an \emph{automatic convergence} theorem for the holomorphic automorphic forms on tube domains in many cases.  We now setup the main results and explain what this means.

\subsection{Holomorphic automorphic forms}
Suppose $J$ is a finite-dimensional Jordan algebra over $\Q$ with $J(\R)$ Euclidean.  One defines the tube domain
\[
\H_J = \{X +iY: X, Y \in J(\R), Y  > 0\}.
\]
Via the Kantor-Koecher-Tits construction, out of $J$ one can create a semisimple Lie algebra $\g_J$ over $\Q$.  Let $G_J$ be a semisimple algebraic group over $\Q$ with Lie algebra $\g_J$.  Then $\H_J$ is the symmetric space associated to the identity component $G_J(\R)^0$.

We prove a theorem about holomorphic modular forms on $\H_J$, or equivalently, about certain automorphic forms on $G_J$.  We now specify the $J$ that we allow.

Suppose $C$ is a rational composition algebra with positive-definite norm.  If $C$ has associative multiplication, then $C$ is either $\Q$, an imaginary quadratic field, or a quaternion algebra ramified at the archimedean place.  If $C$ does not have associative multiplication, then $C = \O$ is the unique rational octonion algebra with positive-definite norm.  Let $H_k(C)$ denote the $k \times k$ Hermitian matrices over the composotion algebra $C$.  We will consider $J$ as in Table \ref{table:globalNotation}.

\begin{table}[h!]
    \centering
    \caption{Global notation}
    \label{table:globalNotation}
\[
\begin{array}{|c|c|c|c|c|c|}
	\hline
	\text{case}&  J  & 1_J & J_1 & r  & e \\
	\hline
1 & H \oplus V_0 & b_{2}+b_{-2} & \Q & 2 & b_{2} = e_{11}\\ 
	\hline
2 & H_n(C) & \sum_{j=1}^{n} e_{jj} & H_{n-1}(C) & n & e_{11} \\
	\hline
3&  H_3(\O) & \sum_{j=1}^{3} e_{jj} & H_2(\O) & 3 & e_{11} \\
\hline
\end{array}
\]
\end{table}

In the second row of Table \ref{table:globalNotation}, $V_0$ denotes a rational quadratic space of dimension at least $1$ with negative definite norm and $H = \Q b_{2} \oplus \Q b_{-2}$ denotes a hyperbolic plane, so that the pairing $(b_2, b_{-2}) =1$.  To make notation consistent between cases, we set $e_{11} = b_2$ and $e_{22} = b_{-2}$.  

The table also defines other notation that will be used below.  In particular, for an integer $1 \leq j \leq k$, $e_{jj}$ denotes the $j^{th}$ diagonal element of $H_k(C)$ and the positive integer $r$ denotes the rank of the Jordan algebra.  In the case of $J = H \oplus V_0$, the identity $1_J$ of $J$ is $1_J = b_2 + b_{-2}$.  In the other cases, $1_J = \sum_{j=1}^{r} e_{jj}$.

In the first case, one can take for the group $G_J$ the special orthogonal group $\SO(V)$, where $V = H^2 \oplus V_0$, or an isogenous group.  In the second case, for $G_J$ one can take 
\[
G_J = \{g \in \GL_{2n,C}: g^* J_n g = J_n\}
\]
where $J_n = \mm{0_n}{1_n}{-1_n}{0_n}$, or an isogenous group.  This covers the case of symplectic groups $\Sp(2n)$, unitary groups $U(n,n)$, and quaternion unitary groups $\SO^*(4n)$, depending on if $C=\Q$, an imaginary quadratic field, or a quaternion algebra ramified at infinity, respectively.  In the third case, for $G_J$ one can take a group that is exceptional of type $E_7$ with real and rational rank equal to $3$.

The algebraic groups $G=G_J$ have a class of holomorphic automorphic forms which we now describe.  Let $P \subseteq G$ denote the Siegel parabolic subgroup of $G$, so that the unipotent radical $N_P$ of $P$ is abelian and can be identified with $J$.  We write
\[ n: J \rightarrow N_P\]
for this isomorphism.  Let $M_P$ denote a standard Levi factor of $P$. There is a factor of automorphy\footnote{The notation ``$J$" for Jordan algebra and factor of automorphy are both so standard that it seems unwise to use a different notation for one or the other.  Hopefully no confusion will arise by using the same notation for both, as the meaning will be clear from context.} 
\[J: G(\R)^0 \times \H_J \rightarrow M_P(\C).\]
Let $1_J \in J$ be the unit of the Jordan algebra so that $i 1_J \in \H_J$.  Let $K^0 \subseteq G(\R)^0$ denote the stabilizer of $i 1_J \in \H_J$.  This is a maximal compact subgroup of $G(\R)^0$.

\begin{definition} Let $\rho$ be a finite-dimensional representation of $M_P(\C)$ on a complex vector space $V_\rho$.   An automorphic function $\varphi: G(\Q)\backslash G(\A) \rightarrow V_\rho$ is said to be a holomorphic automorphic form of weight $\rho$ if, for every $g_f \in G(\A_f)$, the function
\[
g \mapsto \rho(J(g,i 1_J)) \varphi(gg_f)
\]
from $G(\R)^0$ to $V_\rho$ is right-invariant by $K^0$, and the induced function $F_{\varphi,g_f}: \H_J \rightarrow V_\rho$ is holomorphic.  If $U \subseteq G(\A_f)$ is a compact open subgroup, we write $M_{\rho}(U)$ for the space of holomorphic automorphic forms on $G_J$ of weight $\rho$ that are right-invariant by $U$.  Likewise we write $S_{\rho}(U)$ for the subspace of cuspidal forms.
\end{definition}

If $\varphi$ is a holomorphic automorphic form on $G$, then $F_{\varphi,g_f}$ has a classical Fourier expansion:
\[
F_{\varphi,g_f}(Z) = \sum_{T \in J^\vee(\Q), T \geq 0} a_T(g_f) q^T
\]
where $q^T$ is shorthand for $e^{2\pi i (T, g \cdot (i 1_J))}$.  Here $J^\vee$ denotes the linear dual of $J$.  The notation $T \geq 0$ means that $T$ is positive semi-definite in $J^\vee(\R)$.  The functions $a_{T}: G(\A_f) \rightarrow V_\rho$ are locally constant.  One calls $a_T$ the $T^{th}$ Fourier coefficient of $\varphi$.

\emph{For the rest of the paper we now assume that the group of real points $G(\R)$ is connected}, which happens, for example, when $G$ is simply connected.  This is just for simplicity, to avoid making arguments that depend upon component groups.  The Fourier coefficients $a_T$ satisfy relations coming from automorphy along $P$: If $\gamma \in M_P(\Q)$ then
\[
a_T(\gamma_f g_f) = \rho(J(\gamma_\infty, i 1_J)) a_{T \cdot \gamma}(g_f)
\]
for all $g_f \in G(\A_f)$.  Here $\gamma = (\gamma_f,\gamma_\infty) \in M_P(\A_f) \times M_P(\R)$.

The groups $G$ have another conjugacy class of maximal parabolic subgroup.  Let 
\[n_L: J \rightarrow \g_J\]
denote the logarithm of the map $n: J\rightarrow G$ so that $n(x) = \exp(n_L(x))$ for $x \in J$.  The element $n_L(e)$ for $e$ as in Table \ref{table:globalNotation} spans the highest root space for $\g_J$ for a particular choice of minimal parabolic subgroup of $G$.  Let $Q = M_Q N_Q \subseteq G$ denote the stabilizer of the line $\Q e \subseteq \g_J$.  The subgroup $Q$ is a maximal parabolic subgroup of $G$.  The unipotent radical $N_Q$ is a Heisenberg group. Let $Z$ denote its center, so $\mathrm{Lie}(Z) = \Q e$.  Fix a standard additive character $\psi: \Q\backslash \A \rightarrow \C^\times$ with $\psi_\infty(x) = e^{2\pi i x}$.  If $m \in \Q$, define $\varphi_m: G(\A) \rightarrow V_\rho$ as
\[
\varphi_m(g) = \int_{\Q\backslash \A}{\psi(m z)^{-1}\varphi( n(ze) g)\,dz}.
\]
We call this the $m^{th}$ Fourier-Jacobi coefficient of $\varphi$.  It is a holomorphic automorphic Jacobi form of weight $\rho$ and index $m$. See Definition \ref{def:JacobiForms} below. If $g \in G(\R)$ and $g_f \in G(\A_f)$, set 
\[F_{\varphi,g_f,m}(g) = \rho(J(g, i 1_J)) \varphi_m(g g_f).\]
One has
\[
F_{\varphi,g_f,m}(g) = \sum_{T \in J^\vee(\Q), (T,e) = m} a_T(g_f) q^T.
\]
We also write $F_{\varphi,g_f,m}(Z) = F_{\varphi,g_f,m}(g)$ if $g \cdot (i1_J) = Z$.  The function $F_{\varphi,g_f,m}(Z)$ is a holomorphic Jacobi form of weight $\rho$ and index $m$.  Let $Q' \subseteq Q$ denote the centralizer of $Z$.  We have $Q' = N_Q M_Q'$ where $M_Q'$ is the centralizer of $Z$ in $M_Q$. One has $\varphi_m(\gamma g) = \varphi_m(g)$ for all $\gamma \in Q'(\Q)$.

Here is our definition of Jacobi forms.
\begin{definition}\label{def:JacobiForms} Suppose $U$ is an open compact subgroup of $G(\A_f)$ and $Y \subseteq G(\A_f)$ is a subset with $YU = Y$ and $Q'(\A_f) Y = Y$.  A smooth, moderate growth function $\eta: Y \times G(\R) \rightarrow V_{\rho}$ is said to be a holomorphic automorphic Jacobi form of weight $\rho$, index $m$, level $U$, and domain $Y$, written $\eta \in MJ_{\rho,m}(Y,U)$, if it satisfies the following properties:
\begin{enumerate}
    \item $\eta(n(ze) g) = \psi(mz) \eta(g)$ for all $z \in \A$;
    \item For every $g_f \in Y$, the function $F_{g_f}: G(\R) \rightarrow V_{\rho}$ defined as 
    \[F_{g_f}(g) = \rho(J(g,i 1_J)) \eta(g_f g)\]
    descends to $\H_J$ and is holomorphic there;
    \item one has $\eta(hu) = \eta(h)$ for all $u \in U$ and $h \in Y \times G(\R)$;
    \item $\eta(\delta h) = \eta(h)$ for all $\delta \in Q'(\Q)$ and $h \in Y \times G(\R)$.
\end{enumerate}
If the constant term of $\eta$ along the unipotent radical of every standard parabolic subgroup is $0$, we say that $\eta$ is cuspidal.  The subspace of cuspidal forms is denoted $SJ_{\rho,m}(Y,U)$.
\end{definition}

Holomorphic Jacobi forms have a Fourier expansion
\[
F_{g_f}(g) = \sum_{T \in J^\vee(\Q), (T,e) = m} a_T(g_f) q^{T}.
\]

\subsection{Formal modular forms}
The Fourier and Fourier-Jacobi expansions of a holomorphic automorphic form give rise to the following notion, considered by \cite{aoki, IPY, bruinierGenus2, raumGenus2, bruinerRaum2015, bruinierRaum2024, pollackAutConvSMF, AIP, bruinierRaumCorrection} in the case of the symplectic groups.  They are also considered in \cite{wangFreeAlgebrasOMFs} for a handful of orthogonal groups and in \cite{xiaUnitary} for unitary groups whose associated imaginary quadratic field is norm-Euclidean.

For a subgroup $S$ of $G$, denote $S(\Q)_f$ the image of $S(\Q)$ in $S(\A_f) \subseteq G(\A_f)$.
\begin{definition}\label{def:formalMFs} Suppose $U \subseteq G(\A_f)$ is an open compact subgroup, and suppose $X \subseteq G(\A_f)$ is an open subset satisfying $XU = X$ and $G(\A_f) = G(\Q)_{f} X$.   Let $\rho: M_P(\C) \rightarrow \GL(V_{\rho})$ be a finite-dimensional representation.  A collection of functions $a_T: X \rightarrow V_\rho$, one for each $T \in J^\vee(\Q)$, is said to be a \emph{formal modular form of weight $\rho$, level $U$ and domain $X$} if it satisfies the following properties:
\begin{enumerate}
    \item $a_T \equiv 0$ unless $T \geq 0$.
    \item $a_T(xu) = a_T(x)$ for all $T \in J^\vee(\Q), x \in X$ and $u \in U$.
    \item \label{item:propertyPsymms} If $v \in J(\A_f)$, $\gamma \in M_P(\Q)$, $x \in X$ and $n(v) \gamma_f x \in X$ then 
    \[a_T(n(v) \gamma_f x) = \psi((T,v))\rho(J(\gamma_\infty, i 1_J)) a_{T \cdot \gamma}(x).\]
    \item \label{item:propertyQsymms} Let $Y = Q'(\A_f) X$.  For every $m \in \Q$ there is a holomorphic Jacobi form $\eta_m \in MJ_{\rho,m}(Y,U)$ of index $m$, weight $\rho$ and domain $Y$ whose Fourier expansion is
    \[\eta_m(x g) = \rho(J(g, i 1_J))^{-1} \sum_{T \in J^\vee(\Q), (T,e) = m}a_T(x) q^T\]
    if $x \in X$ and $g \in G(\R)$.
    \end{enumerate}
One writes $M_\rho^{f}(X,U)$ for the $\C$-vector space of formal modular forms of weight $\rho$, domain $X$ and level $U$.
\end{definition}

Evidently, the map $\varphi \mapsto (a_T)_T$ defines an injective linear map
\[M_\rho(U) \rightarrow M_\rho^{f}(X,U)\]
from the holomorphic automorphic forms of weight $\rho$ and level $U$ to the formal modular forms of this weight and level, for any open $X$ with $G(\A_f) = G(\Q)_fX$.

Here is the main theorem of this paper. Let $R \subseteq G$ denote the parabolic subgroup defined as $P \cap Q$, and set $R' = P \cap Q'$.  Note that there are compact open subsets $X \subseteq G(\A_f)$ satisfying $P(\Q)_fX = G(\A_f) = Q(\Q)_f X$, see Lemma \ref{lem:NPMQKf} and likewise $Q'_f(\Q) X \subseteq N_P(\A_f) R'(\Q)_f X$.
\begin{theorem}\label{thm:autConvIntro} Let $U \subseteq G(\A_f)$ be a compact open subgroup and suppose $X \subseteq G(\A_f)$ is a subset satisfying 
\begin{enumerate}
    \item $XU = X$;
    \item $P(\Q) X= G(\A_f) = Q(\Q) X$;
    \item $Q'(\Q) X \subseteq N_P(\A_f) R'(\Q) X$.
\end{enumerate} Then the map $M_\rho(U) \rightarrow M_\rho^{f}(X,U)$ is a $\C$-linear isomorphism.
\end{theorem}

Note that, because an honest holomorphic modular form $\varphi$ is of moderate growth on $G(\A)$, its Fourier coefficients $a_T(g_f)$ inherit a polynomial growth property in $T$ for every $g_f \in G(\A_f)$.  The main difficulty in proving Theorem \ref{thm:autConvIntro} is proving that if $(a_T)_T \in M_\rho^{f}(X,U)$, then the $a_T$ necessarily also grow polynomially in $T$.  For this reason, we call Theorem \ref{thm:autConvIntro} the \emph{automatic convergence theorem}: The symmetries of the Fourier coefficients of a hypothetical holomorphic modular form force the infinite sum $\sum_{T} a_T(g_f) q^T$ to converge automatically.

\begin{remark}\label{rmk:smallX?} In case $X \subseteq G(\A_f)$ satisfies $G(\A_f) = G(\Q)_{f} X$, but not necessarily $G(\A_f) = N_P(\A_f) M_P(\Q) X$ (for example, if $X=U$ is an arbitrary open compact subgroup of $G(\A_f)$ and $G$ is simply connected), one could still hope to prove that the map $M_\rho(U) \rightarrow M_\rho^{f}(X,U)$ is an isomorphism.  We do not know if this is true and seems to be an interesting question in general.  Indeed, as described below, our arguments rely heavily on reduction theory and Siegel sets.  But only for large congruence subgroups $\Gamma$ of a semisimple group $H(\R)$ is it true that every coset $\Gamma \backslash H(\R)$ has a representative in a Siegel set.  In trying to shrink $X$ in Theorem \ref{thm:autConvIntro}, one must find a way to get around this basic constraint, which we do not know how to do.
\end{remark}

\subsection{Related work} The study of formal modular forms and their modularity goes back to \cite{aoki}, and then \cite{IPY,bruinierGenus2,raumGenus2} in the case of $\Sp_4$.  For the symplectic groups $\Sp_{2n}$, formal modular forms were studied in \cite{bruinerRaum2015, bruinierRaumCorrection} and in our arXiv preprint \cite{pollackAutConvSMF}.  In these papers, the analogue of the subset $X$ of Theorem \ref{thm:autConvIntro} is equal to or of small index in the maximal compact subgroup $\Sp_{2n}(\widehat{\Z})$.  Likewise in \cite{wangFreeAlgebrasOMFs} and the paper \cite{xiaUnitary}, which prove results about formal modular forms on eight different orthgonal groups and, respectively, on unitary groups over a norm-Euclidean imaginary quadratic field.

In the recent works \cite{bruinierRaum2024} and \cite{AIP} smaller $X$ are considered, which force new techniques to be used.  See Remark \ref{rmk:smallX?}.  We also mention the recent works \cite{flores} and \cite{fanCohDimSMV} which consider formal modular forms on the symplectic group from an algebro-geometric perspective.

With the exception of \cite{pollackAutConvSMF}, the proofs of the modularity of formal modular forms in the above-mentioned papers rely substantially on complex geometry or precise dimension formulas for holomorphic modular and Jacobi forms in some cases.  In our work \cite{pollackAutConvExc}, we developed an analogous notion of formal modularity for the non-holomorphic quaternionic modular forms on certain exceptional groups of type $F_4, E_{6}, E_{7}$, and $E_8$.  The proof of the modularity in \cite{pollackAutConvExc} and the companion note \cite{pollackAutConvSMF} rely on reduction theory and a quantitative Sturm bound, instead of geometry or complex analysis.

\subsection{Proof idea} We now sketch the proof idea of Theorem \ref{thm:autConvIntro} in case the formal modular form $(a_T)_{T \in J^\vee(\Q)}$ is cuspidal.  The general case follows from the cuspidal case by an application of the Koecher principle, see section \ref{sec:autConv}.

A key tool used is a quantitative Sturm bound for Jacobi forms.  The quantitative Sturm bound says that if the ``first" several Fourier coefficients of some cuspidal holomorphic Jacobi form $\eta$ are small, then all the Fourier coefficients of $\eta$ are correspondingly small.  

Suppose now one is given $(a_T)_{T \in J^\vee(\Q)} \in M_{\rho}^{f}(X,U)$, and assume that $a_T \neq 0$ implies $T$ is positive-definite so that we are in the cuspidal case.  For simplicity also assume that $\rho$ is one-dimensional so that we are considering scalar weight modular forms.  We sketch the proof that $|a_T(g_f)|$ grows polynomially in $T$ for every $g_f \in G(\A_f)$.  

Without loss of generality, by shrinking $X$ if necessary, we can assume $X \subseteq G(\A_f)$ is compact open.  By property \ref{item:propertyPsymms} of Definition \ref{def:formalMFs}--called the ``$P$-symmetries"--it suffices to bound $|a_T(x)|$ for the finitely many $x \in X/U$.  Assume for simplicity that $X = 1 \cdot U$.  Given $M > 0$ and $\delta > 1$, for some positive integer $k$ one has $\det(T) \leq M \cdot \delta^{k+1}$; we induct on $k$.  Assume one has an appropriate bound on the $|a_{S}(1)|$, for those $S$ with $\det(S) \leq M \cdot \delta^{k}$, and we'd like to bound $|a_{T}(1)|$ assuming $\det(T) \leq M \cdot \delta^{k+1}.$  By the $P$-symmetries, to bound $|a_T(1)|$, it suffices to bound $|a_{T \cdot \gamma}(1)|$ for some $\gamma$ in the congruence subgroup $\Gamma_{P,U}:= M_P(\Q) \cap U$.  

By reduction theory for $\Gamma_{P,U}$ acting on $J^\vee(\Q)$, up to some controllable finite discrepancy, one can choose $\gamma$ so that $T':=T \cdot \gamma$ satisfies $t':=(T',e)$ is relatively small compared to $\det(T)$.  Now using property \ref{item:propertyQsymms} of Definition \ref{def:formalMFs}--called the ``$Q$-symmetries"--for $t'$, one sees the coefficient $a_{T'}(1)$ in the Jacobi form $F_{t',1}(Z)$.  The ``first" Fourier coefficients of $F_{t',1}$ are now all of the form $a_{S}(1)$ for $\det(S) \leq M \cdot \delta^{k}$.  Using the inductive hypothesis--i.e., the previously-established bound on the $|a_{S}(1)|$ for $\det(S) \leq M \cdot \delta^{k}$--one can apply the quantitative Sturm bound to deduce the desired bound on $|a_T(1)|$.

\subsection{Outline of paper} We now outline the rest of the paper.  In section \ref{sec:holMFs} we describe the action of $G_J(\R)$ on $\H_J$ and also prove some simple facts about holomorphic modular forms on $G_J$. In section \ref{sec:reductionTheory} we review the classical reduction theory of Borel and Harish-Chandra and derive some consequences for later application.  In section \ref{sec:QSBs} we prove a quantitative Sturm bound for holomorphic cuspidal Jacobi forms.  In section \ref{sec:autConv} we prove the automatic convergence theorem.

\section{Tube domains and holomorphic modular forms}\label{sec:holMFs}
In this section we prove some results about the tube domain $\H_J$, the action of $G_J(\R)$ on it, and about the Fourier coefficients of holomorphic modular forms on $\H_J$.

\subsection{The Jordan algebra}
We begin with some notation for the Jordan algebra.  

Recall that $r$ denotes the rank of $J$.  There is a homogeneous degree $r$ polynomial $\det \in Sym^{r}(J^\vee)$ called the determinant or Jordan norm.  It induces a homogeneous degree $r$ polynomial map 
\[
\det: J(R) \rightarrow R
\]
for any $\Q$-algebra $R$.  The determinant satisfies $\det(1_J) = 1$, and the action of $M_P$ on $J$ preserves the determinant up to scaling.  Likewise, there is $\det \in Sym^{r}(J)$, which induces
\[
\det: J^\vee(R) \rightarrow R
\]
for any $\Q$-algebra $R$.  We denote it by the same symbol.  For the $J$ of Table \ref{table:globalNotation}, the pairing on $J$ given by $(X,Y) = \frac{1}{2}\tr(XY+YX)$ induces an identification $J \rightarrow J^\vee$.  Via this identification, the determinant on $J^\vee$ is the same as that on $J$.

In cases 2 and 3 of Table \ref{table:globalNotation}, let $J_1 \subseteq J$ denote the subspace
\[
J_1 = \{X \in H_n(C): X_{1j} = 0=X_{1j} \text{ for } 1 \leq j \leq r\}.
\]
Thus $J_1$ can be identified with $H_{n-1}(C)$ in these cases.  In these cases, let $V_0 \subseteq J$ be the subspace
\[
V_0 = \{X \in H_n(C): X_{11} = 0 \text{ and } X_{ij} = 0 = X_{ji} \text{ if } i \geq 2\}.
\]
Thus $V_0 \simeq C^{n-1}$.  In case 1 of Table \ref{table:globalNotation}, let $J_1 = \Q b_{-2} = \Q e_{22}$.  In all cases, one has a direct sum decomposition
\begin{equation}\label{eqn:J_directSum3}
	J = \Q e_{11} \oplus V_0 \oplus J_1.
\end{equation}

The Jordan algebra $J(\R)$ has a notion of positive-definiteness, defined as follows.
\begin{definition} An element $Y \in J(\R)$ is said to be \emph{positive-definite} if there exists $h \in M_P(\R)$ with $Y = h \cdot 1_J$.  In this case, one writes $Y >0$.  If $Y_1, Y_2$ are in $J(\R)$ one writes $Y_1 > Y_2$ if $Y_1-Y_2 > 0$.   Likewise $T \in J^\vee(\R)$ is said to be positive-definite if there exists $h \in M_P(\R)$ so that $T = 1_J \cdot h$.
\end{definition}

We will use the following lemma below.
\begin{lemma}\label{lem:detIneq3part} Suppose $Y \in J(\R)$ is positive-definite, $Y = (\alpha, v, Y_1)$ in the notation of the decomposition \ref{eqn:J_directSum3}.  Then $\det(Y) \leq \det((\alpha,0,Y_1))$.
\end{lemma}

\subsection{The action on the tube domain}
Let $\g_J$ denote the rational Lie algebra of $G_J$, $\g_0 = \g_J \otimes_\Q \R$ and $\g = \g_0 \otimes_{\R} \C$.   Let $\theta$ be the Cartan involution on $\g_0$ and $B: \g_J \times \g_J \rightarrow \Q$ a $G$-invariant symmetric bilinear form so that 
\[
B_{\theta}(X,Y):= - B(X,\theta(Y))
\]
is symmetric and positive-definite on $\g_0$.  Let $K \subseteq G(\R)$ be the subgroup of $G(\R)$ preserving the bilinear form $B_{\theta}$.  Then $K$ is a maximal compact subgroup of $G(\R)$.

Let $P^{op} = N_{P^{op}} M_P$ denote the parabolic subgroup of $G$ opposite to $P$, and let $\n_{P^{op}}$ denote the Lie algebra of $N_{P^{op}}$.  Likewise let $\n_{P}$ denote the Lie algebra of $N_P$ and let $\m_{P}$ denote the Lie algebra of $M_P$. One has a direct sum decomposition
\[
\g_J = \n_{P^{op}} \oplus \m_P \oplus \n_{P}.
\]

Recall the $M_P$-equivariant linear map $n_L: J \rightarrow \n_{P}$.  One also has an $M_P$-equivariant linear map $n_{L}^{op}: J^\vee \rightarrow \n_{P^{op}}$. Set $\g = \g_J \otimes \C$.  There is a polynomial map
\[
r: J \rightarrow J \otimes \g_J
\]
given by
\[
Z \mapsto \sum_{\alpha} X_{\alpha} \otimes (Ad(n(Z)) \cdot n_{L}^{op}(X_{\alpha}^\vee)).
\]
Here $X_{\alpha}$ is a basis of $J$ and $X_{\alpha}^\vee$ is the dual basis of $J^\vee$. The expression $Ad(n(Z))$ denotes the adjoint action of $G_J$ on $\g_J$.  Extending scalars to $\C$ gives a map
\[
r: J(\C) \rightarrow J(\C) \otimes_{\C} \g
\]
by the same formula.  We let $G(\R) \subseteq G(\C)$ act on $J(\C) \otimes_{\C} \g$ via its adjoint action on the second factor.  If $m \in M_P(\C)$, then we write $L(m)$ for the action of $m$ on the first factor $J(\C)$.

Recall that
\[
\H_J = \{X + iY: X,Y \in J(\R), Y>0\}.
\]
The action of $G(\R)$ on $\H_J$ and the factor of automorphy $J: G(\R) \times \H_J \rightarrow M_P(\C)$ can be defined simultaneously as a result of the following lemma.

\begin{lemma}\label{lem:almostAction} Suppose $g \in G(\R)$, $Z \in \H_J$ and $k \in K$.
\begin{enumerate}
    \item One has $Ad(k) r(i 1_J) = L(m_k)^{-1} r(i 1_J)$ for some $m_k \in M_P(\C)$.
    \item In general, there exists a unique $m' \in M_P(\C)$ and $Z' \in \H_J$ so that
\[
Ad(g) \cdot r(Z) = L(m')^{-1} r(Z').
\]
\end{enumerate}
\end{lemma}
\begin{proof} The lemma is essentially the Harish-Chandra embedding 
\begin{equation*}\label{eqn:HCE}
G(\R) \subseteq n(J(\C))P^{op}(\C)
\end{equation*}
phrased differently.  See, for example, \cite[section 3]{wolfHermSym}.

The first part is essentially the Cayley transform.  For the second, we have $g n(Z) = n(Z') m' v'$ inside of $G(\C)$ for some $Z' \in J(\C)$, $m' \in M_P(\C)$  and $v' \in N_{P^{op}}(\C)$.  The adjoint action of $v'$ on the $n_{L}^{op}(X_\alpha^{\vee})$ is trivial, so one obtains $Ad(g) \cdot r(Z) = L(m')^{-1} r(Z')$.  

To see that $Im(Z') > 0$, one can argue as follows.  There is $p \in P(\R)$ so that $Ad(p) r(i 1_J) = L(m)^{-1} r(Z)$.  Thus $L(m)^{-1} Ad(g) r(Z) = Ad(gp) r(i 1_J)$.  By the Iwasawa decomposition, $gp = p' k'$ for some $p' \in P(\R)^{0}$ and $k' \in K$. If $p' = n(x')m'$ with $x' \in N_P(\R)$ and $m' \in M_P(\R)^{0}$, then one sees $Z' = x'+ i (m' \cdot 1_J)$, so that $Y' = m' \cdot 1_J$ is positive-definite.
\end{proof}

\begin{definition} Suppose $g \in G(\R)^{0}$ and $Z\in \H_J$. Applying Lemma \ref{lem:almostAction}, $Ad(g) r(Z) = L(m')^{-1} r(Z')$ for a unique $m' \in M_P(\C)$ and $Z' \in \H_J$.  One sets $gZ = Z'$ in $\H_J$ and $J(g,Z) = m'$.  This defines an action
\[
G(\R) \times \H_J \rightarrow \H_J
\]
and the factor of automorphy 
\[
J: G(\R) \times \H_J \rightarrow M_P(\C).
\]
The factor of automorphy and the action satisfy
\[
J(g_1 g_2, Z) = J(g_1, g_2 Z) J(g_2, Z)
\]
for every $Z \in \H_J$, and $g_1, g_2 \in G(\R)$.
\end{definition}
Note that the map $k \mapsto J(k,i1_J)$ from  $K$ to $M_P(\C)$ is a group homomorphism.  Also, if $m \in M_P(\R)$ then 
\[J(m, Z) = m \in M_P(\R) \subseteq M_P(\C)\]
independent of $Z \in \H_J$.

Set $K_M = M_P(\R) \cap K$.  One has
\[K_M = \{k \in M_P(\R): 1_J \cdot k = 1_J\}.\]

\subsection{Fourier coefficients of holomorphic modular forms}
We consider the Fourier coefficients of holomorphic automorphic forms on $G$.  

Suppose $\rho: M_P(\C) \rightarrow \GL(V_{\rho})$ is a finite-dimensional complex representation.  Fix once and for all a $J(K,i 1_J)$-invariant norm $|| \cdot ||$ on $V_\rho$.

\begin{definition}\label{defn:betaTuT} Suppose $T \in J(\Q)^\vee$ is positive-definite.  Let $u_T \in M_P(\R)$ be such that $1_J \cdot u_T = T$; $u_T$ is unique up to left-multiplication by an element of $K_M$.  Suppose $\varphi \in M_{\rho}(U)$ is a holomorphic automorphic form of weight $\rho$ and let $a_T: G(\A_f) \rightarrow V_{\rho}$ denote its Fourier coefficients.  We set 
	\[\beta_{T}(g_f) = || \rho(J(u_T,i 1_{J})) a_{T}(g_f)||.\]
Note that $\beta_T$ is independent of the choice of $u_T$, and satisfies $\beta_{T}(m_f g) = \beta_{T \cdot m}(g)$ for all $m \in M_P(\Q)$ and $g \in G(\A_f)$.
\end{definition}

The following lemma bounds the size of the $\beta_T$ for $\varphi$ a cuspidal holomorphic automorhpic form $\varphi$ on $G$ or a cuspidal holomorphic Jacobi form.  For such a $\varphi$, we set $||\varphi||$ the supremum of $||\varphi(g)||$ over $g \in G(\A)$.  If $Y \subseteq G(\A_f)$ is a subset with $N_P(\A_f) Y = Y$, we set $||\varphi||_Y$ the supremum of $||\varphi(g)||$ over $g \in Y \times G(\R)$.
\begin{lemma} Suppose $Y \subseteq G(\A_f)$ satisfies $N_P(\A_f) Y = Y$ and suppose $\varphi: Y \times G(\R) \rightarrow \C$ is continuous and left-invariant by $N_P(\Q)$.  Assume moreover that $g\mapsto J(g,i1_J) \varphi(yg)$ descends to a holomorphic function on $\H_J$ for every $y \in Y$. Then $\beta_{T}(y) \leq e^{2\pi r} ||\varphi||_{Y}$ for all $y \in Y$.
\end{lemma}
\begin{proof} Recall that $\psi: \Q\backslash \A \rightarrow \C^\times$ denotes our fixed additive character.  For $T \in J^\vee(\Q)$ and $h \in Y \times G(\R)$, let 
\[
\varphi_T(h) = \int_{N_P(\Q)\backslash N_P(\R)}\psi^{-1}((T,x))\varphi(n(x)h)\,dx.
\]
We have 
\[\rho(J(g,i1_J)) \varphi_T(yg) = a_T(y) e^{2\pi i (T,g \cdot i 1_J)}.\]

For any $g \in G(\R)$ and $y \in Y$, one has 
\begin{align*}
||\varphi||_Y &\geq ||\varphi_T(y g)|| \\
&= || e^{2\pi i (T,g \cdot (i 1_J))} \rho(J(g,i))^{-1} a_T(y)|| \\
&=  || e^{2\pi i (T,g \cdot (i 1_J))} \rho(J(u_Tg,i))^{-1} \rho(J(u_T,i 1_J)) a_T(y)||.
\end{align*}
Now take $g = u_T^{-1}$.
\end{proof}

\section{Reduction theory}\label{sec:reductionTheory}
In this section we recall the reduction theory of Borel and Harish-Chandra and derive some specific consequences of it for our later application.

\subsection{Basic reduction theory}
Suppose $H$ is a reductive algebraic group over $\Q$ and $R\subseteq H$ is a parabolic subgroup.  Let $\Phi_R$ be the set of roots of $H$ contained in the unipotent radical of $R$.  For $\epsilon > 0$, define
\[
R_{\epsilon} = \{r \in R(\R): |\alpha(r)| \geq \epsilon \text{ for all } \alpha \in \Phi_R\}.
\]
Let $K_H \subseteq H(\R)$ be a maximal compact subgroup with $H(\R) = R(\R)K_H$.  If $M_R$ is a Levi factor of $R$ compatible with the Cartan involution defining $K_H$, let 
\[
M_{R,\epsilon} = \{r \in M_R(\R): |\alpha(r)| \geq \epsilon \text{ for all } \alpha \in \Phi_R\}.
\]
Set $S_\epsilon = R_\epsilon K_H$.

\begin{theorem}[Borel, Harish-Chandra]\cite{borelHarishC}\label{thm:BHCred} Suppose $H$ is a reductive algebraic group over $\Q$ and $U_H \subseteq H(\A_f)$ is a compact open subgroup.  Set $\Gamma_U = H(\Q) \cap U$, thought of as a subgroup of $H(\R)$.  Let $R \subseteq H$ be a parabolic subgroup and $K_{H} \subseteq H(\R)$ a maximal compact subgroup with $H(\R) = R(\R) K_H$.  Then, there exists a finite set $F \subseteq H(\Q)$ and a sufficiently small $\epsilon>0$ so that $H(\R) = \Gamma_U F S_{\epsilon}$.  In fact, there exists a compact subset $U_R^C$ of the unitpotent radical $U_R(\R)$ of $R(\R)$ so that $H(\R) = \Gamma_U F U_R^C M_{R,\epsilon} K_H$.
\end{theorem}

One can make Theorem \ref{thm:BHCred} adelic.  This is well-known.  We write down the details for the convenience of the reader.
\begin{corollary}\label{cor:reductionAdelic} Suppose $H$ is a reductive algebraic group over $\Q$.  Let $R \subseteq H$ be a parabolic subgroup and $K_{H} \subseteq H(\R)$ a maximal compact subgroup with $H(\R) = R(\R) K_H$.  Then there is a compact open subset $X$ of $H(\A_f)$ and a sufficiently small $\epsilon > 0$ so that $H(\A) = H(\Q) S_{\epsilon} X$.
\end{corollary}
\begin{proof} Let $X_0 \subseteq H(\A_f)$ be compact open such that $H(\A_f) = H(\Q) X_0$.  Let $U \subseteq H(\A_f)$ be a compact open subset satisfying $U X_0 = X_0$.  Let $\Gamma_U, F, \epsilon$ be as in Theorem \ref{thm:BHCred}.  Set $X = F^{-1} X_0$.

If $h_f \in H(\A_f)$, then $h_f = \mu_f x_0$ for some $\mu \in H(\Q)$ and $x_0 \in X_0$.  Now suppose $h \in H(\R)$.  Then $\mu_\infty^{-1} h = \gamma_\infty t_\infty s$ for some $\gamma_\infty \in \Gamma_U$, $t_\infty \in F$ and $s \in S_{\epsilon}$.    Now
\begin{align*}
h_f h_\infty &= \mu_f x_0 h_\infty \\
&= \mu \mu_\infty^{-1} h_\infty x_0 \\
&= \mu \gamma_\infty t_\infty s x_0 \\
&= \mu \gamma t s (t_f^{-1} \gamma_f^{-1} x_0).
\end{align*}
But now $\mu, \gamma, t \in H(\Q)$.  One has $\gamma_f^{-1}\in U$ and $UX_0 = X_0$, so $t_f^{-1} \gamma_f^{-1} x_0 \in F^{-1} X_0 = X$.
\end{proof}

Let $M_P'$ denote the subgroup of $M_P$ fixing the determinant on $J$ or equivalently fixing the determinant on $J^\vee$. We will apply Theorem \ref{thm:BHCred} to the action of $M_P'(\R)$ on $J^\vee(\R)$.  We obtain the following.
\begin{proposition}\label{prop:redMP} Fix $e \in J$ as in Table \ref{table:globalNotation} and a compact open subgroup $U'$ of $M_P'(\A_f)$.  Let $\Gamma_{P,U'} = M_P'(\Q) \cap U'$.  Then there is a constant $B_{e,U'} > 0$ and a finite subset $F_{e,U'}$ of $M_P'(\Q)$ with the following property: if $T \in J(\R)^\vee$ is positive-definite, then there exists $\gamma \in \Gamma_{P',U'}$ and $t \in F_{e,U'}$ so that $|(T \cdot \gamma t, e)| \leq B_{e,U_M} \det(T)^{1/r}.$
\end{proposition}
\begin{proof}
Let $R_e \subseteq M_P'$ be the parabolic subgroup stabilizing the line spanned by $e$ in the action of $M_P'$ on $J$.  Let $\alpha_e: R_e \rightarrow \GL_1$ be the map given by $r e  = \alpha_e(r) e$ for $r \in R_e$.  For $\epsilon > 0$, set $R_e(\epsilon) = \{r \in R_e(r): |\alpha_e(r)| \geq \epsilon\}$.  By Theorem \ref{thm:BHCred}, there is $\epsilon_M > 0$ sufficiently small and a finite subset $F_M$ of $M_P'(\Q)$ so that
\[M_P'(\R) = \Gamma_{P,U'} F_M R_e(\epsilon_M) K_M.\]

Now, suppose $T \in J^\vee(\R)$ is positive-definite, set $\lambda = \det(T)$, and set $T_1 = \lambda^{-1/r} T$ so that $\det(T_1) = 1$.  Then there exists $m_1 \in M_P'(\R)$ with $1_J \cdot m_1^{-1}= T_1$.  We can write $m_1 = \gamma t m_2 k$ with $\gamma \in \Gamma_{P',U'}$, $t \in F_M$, $m_2 \in R_e(\epsilon_M)$ and $k \in K_M$.  Then
\[
\lambda^{1/r} (1_J \cdot m_2^{-1}) = T \cdot (\gamma t).
\]
Moreover,
\[(T \cdot \gamma t, e) = \lambda^{1/r} (1_J \cdot m_2^{-1}, e) = \det(T)^{1/r} \alpha_e(m_2^{-1}).\]
Thus $|(T \cdot \gamma t, e)| \leq \det(T)^{1/r} \epsilon_M^{-1}$.  This proves the proposition.
\end{proof}

We will also need some reduction theory for the action of the Jacobi group $Q'$ on $G(\A)$.  Recall that $Q' = N_Q M_Q' \subseteq Q$ where $M_Q'$ is the centralizer of $Z$ in $M_Q$.  By Corollary \ref{cor:reductionAdelic} applied to the parabolic subgroup $S_Q:=P \cap M_Q'$ of $M_Q'$, we have 
\begin{equation}\label{eqn:MQ'decomp} 
M_Q'(\A) = M_Q'(\Q) S_Q(\epsilon_J) X_J
\end{equation} 
for some compact open subset $X_J \subseteq M_Q'(\A_f)$ and some $\epsilon_J > 0$.

Recall that $\theta: \g_0 \rightarrow \g_0$ denotes the Cartan involution.  Set $h_{e} = -[e,\theta(e)] \in \g_0$, so that $e, h_{e}, -\theta(e)$ form an $\sl_2$-triple.  Define $\lambda: \R^\times_{>0} \rightarrow M_Q(\R)$ as $\lambda(t) = \exp(\log(t) h_e)$.   Note that $\lambda(t)$ lands in the center of $M_Q(\R)$.

\begin{lemma}\label{lem:JacobiReduction}  Suppose $X \subseteq G(\A_f)$ is some open subset and $Y = Q'(\A_f) X$.  With notation as above, we have
\[
Y \times G(\R) = Q'(\Q) N_Q(\R) S_Q(\epsilon_J) \lambda(\R^\times_{>0}) K \cdot (X_JX).
\]
Moreover, suppose $U_1 \subseteq N_Q(\A_f)$ is an open compact subgroup with the property that $U_1 (X_J X) = X_J X$.  Then there is a compact subset $N_Q^{U_1}$ of $N_Q(\R)$, depending only on $U_1$, so that 
\[
Y \times G(\R) = Q'(\Q) N_Q^{U_1} S_Q(\epsilon_J) \lambda(\R^\times_{>0}) K \cdot (X_JX).
\]
\end{lemma} 
\begin{proof} We have $y = n_f m_f x$ with $n_f \in N_Q(\A_f)$, $m_f \in M_Q(\A_f)$ and $x\in X$.  By the Iwasawa decomposition we have $g = n m t k$ with $n \in N_Q(\R)$, $m \in M_Q'(\R)$, $t \in \lambda(\R^\times_{>0})$ and $k \in K$. Thus $yg = n n_f (m_f m) tk x$.  By reduction theory for $m_f m \in M_Q'(\A)$, we have $m_f m = \gamma_\Q m_1 x_1$ with $\gamma_\Q \in M_Q'(\Q)$, $m_1 \in S_Q(\epsilon_J)$ and $x_1 \in X_J$.  Thus
\[
yg = n n_f (\gamma_\Q m_1) tk x_1 x = \gamma_\Q n' n_f' m_1 t k (x_1 x).
\]
Let $U'$ be a sufficiently small open compact subgroup of $N_Q(\A_f)$.  By strong approximation for the additive group $\A_f$, we have $n_f' = n_\Q n_\infty^{-1} u'$ for some $u' \in U$.  Thus
\[
yg = \gamma_\Q n' (n_\Q n_\infty^{-1} u') m_1 tk (x_1 x) = \gamma_\Q n_\Q (n_\infty^{-1} n') m_1 t k (u' x_1 x).
\]
Here $\gamma_\Q n_\Q \in Q'(\Q)$, $n_\infty^{-1} n' \in N_Q(\R)$, $m_1 \in S_Q(\epsilon_J)$, $t \in \lambda(\R^\times_{>0})$ and $k \in K$.  Since $X$ is assumed open, there is $U' \subseteq N_Q(\Q_f)$ sufficiently small so that $U' x_1 X \subseteq x_1 X$.  This completes the proof of the first statement.  The second statement follows immediately from the first.
\end{proof}

\section{Quantitative Sturm bound}\label{sec:QSBs}
In this section we prove a quantitative Sturm bound for cuspidal holomorphic Jacobi forms.  Recall that we write $SJ_{\rho,m}(Y,U)$ for the cuspidal Jacobi forms of weight $\rho$, index $m$, domain $Y$ and level $U$.

\begin{lemma}\label{lem:NPMQKf} Fix an open compact subgroup $K_f$ of $G(\A_f)$.  There is a finite subset $F$ of $G(\A_f)$ so that $G(\A_f) = P(\Q)_f F K_f$ and $G(\A_f) = Q(\Q)_f F K_f.$
\end{lemma}
\begin{proof} By the Iwasawa decomposition, there is a finite subset $F_0$ of $G(\A_f)$ so that $G(\A_f) = P(\A_f) F_0 K_f$.  Let $K' = \cap_{x_0 \in F_0}{x_0 K_f x_0^{-1}}$, and set $U' = M_P(\A_f) \cap K'$.  By weak approximation for $M_P$, one has a finite subset $F_1$ of $M_P(\A_f)$ so that $M_P(\A_f) = M_P(\Q)_f F_1 U'$.  Setting $F = F_1 \cdot F_0$ gives that $G(\A_f) = N_P(\A_f) M_P(\Q) F X$.  

Now, take an open compact subgroup $U_1$ of $G(\A_f)$ satisfying $U_1 (F K_f) = F K_f$.  Suppose $g_f \in G(\A_f)$ and $g_f = n m y$ with $n \in N_P(\A_f)$, $m \in M_P(\Q)_f$ and $y \in F K_f$.  Let $V_1 = N_P(\A_f) \cap m U_1 m^{-1}$.  This is an open compact subgroup of $N_P(\A_f)$.  By strong approximation for unipotent groups, $N_P(\A_f) = N_P(\Q)_f V_1$.  The result for $P$ follows, and the proof for the parabolic subgroup $Q$ is similar.
\end{proof}

Fixing a compact open subset $X$ of $G(\A_f)$, the next lemma shows that the $T \in J^\vee(\Q)$ for which $a_T$ can be nonvanishing on $X$ lie in some lattice.

\begin{lemma}\label{lem:WhitTsupport} Suppose $U$ is an open compact subgroup of $G(\A_f)$ and $X \subseteq G(\A_f)$ is an open compact subset with $X U = X$.  There is a lattice $J_0 \subseteq J^\vee(\Q)$ depending only on $X$ and $U$ with the following property: If $a_{T}: N_P(\A_f) X \rightarrow \C$ is a right $U$-invariant function satisfying $a_T(n(x) r) = \psi((T,x)) a_T(r)$ for all $x \in J(\A_f)$ and $r \in X$, then $a_T(r) \neq 0$ for some $r \in X$ implies $T \in J_0$.
\end{lemma}
\begin{proof} We can write $X$ as a finite disjoint union $X = \bigsqcup_{j=1}^{H} r_j U$.  Let $V = \bigcap_{j=1}^{H} r_j U r_j^{-1}$, which is an open compact subgroup of $G(\A_f)$.  The set of $x \in J(\A_f)$ with $n(x) \in V$ is an open compact subgroup $W_0$ of $J(\A_f)$.  If $w \in W_0$, then
\[
\psi((T,w)) a_T(r_j) = a_T(n(w) r_j) = a_T(r_j (r_j^{-1} n(w) r_j)) = a_T(r_j)
\]
because $r_j^{-1} n(w) r_j \in U$.  Thus $a_T(r_j)  \neq 0$ for some $j$ implies $(T,w) \in \widehat{\Z}$ for all $w \in W_0$.  Since $T \in J^\vee(\Q)$ is rational, this implies $T$ sits in some lattice $J_0$ of $J^\vee(\Q)$.
\end{proof}

To estimate certain sums that appear in the proof of the quantitate Sturm bound, we will use the following two lemmas.  For a positive number $t$, let
\[
Q_{t} :=\{T \in J^\vee(\R): T \geq 0, (T,1_J) \leq t\}.
\]
\begin{lemma}\label{lem:J0LatticeCountTr} For $t > 0$, the set $Q_t = t \cdot Q_{1}$ is compact.  If $J_0 \subseteq J^\vee(\R)$ is a fixed lattice, then $\#(J_0 \cap Q_t)$ is $O(t^{\dim(J)})$.
\end{lemma}
\begin{proof}  Define a map $[0,1]^r \times K_M \rightarrow J^\vee(\R)$ as 
\[((t_1, \cdots, t_r), k) \mapsto (t_1 e_{11} + \cdots + t_r e_{rr}) \cdot k.\]
By the spectral theorem \cite[Theorem III.1.2 and Corollary IV.2.7]{farautKoranyi}, $Q_1$ is contained in the image of the map.  As $Q_t = t \cdot Q_1$ is closed, the compactness follows.  The estimate on the number of lattice points also follows.
\end{proof}

\begin{lemma}\label{lem:deltaSum} Fix a lattice $J_0 \subseteq J^\vee(\Q)$.  Then there is a constant $C_{J_0} > 0$ so that for every real number $\delta > 0$ one has
\[
\sum_{T \in J_0, T > 0}e^{-\delta (T,1)} < \frac{C_{J_0}}{\delta^{\dim(J)+1}}.
\]
\end{lemma}
\begin{proof} There is a positive integer $M$ so that every element of $J_0$ has trace landing in $M^{-1} \Z$.  By Lemma \ref{lem:J0LatticeCountTr}, there is $C_{1,J_0} > 0$ so that if $n$ is a positive integer, the number of positive-definite elements of $J_0$ with trace equal to $\frac{n}{M}$ is at most $C_{1,J_0} n^{\dim(J)}$.  We obtain
\begin{align*}
\sum_{T \in J_0, T > 0}e^{-\delta (T,1)} &\leq \sum_{n \geq 0} C_{1,J_0} n^{\dim(J)} e^{-\delta n/M} \\
&\leq C_{1,J_0} \delta^{-\dim(J)} \sum_{n \geq 0} \left((n\delta)^{\dim(J)}e^{-\delta n/(2M)}\right) e^{-\delta n/(2M)} \\
&\leq C_{1,J_0} C_M \delta^{-\dim(J)} \left(\sum_{n\geq 0}e^{-\delta n/(2M)}\right) \\
&\leq C_{1,J_0} C_M \delta^{-\dim(J)}\frac{1}{1-e^{-\delta/(2M)}} \\
&\leq C_{J_0} \delta^{-(\dim(J)+1)}.
\end{align*}
Here $C_M > 0$ is such that $x^{\dim(J)} e^{-x/(2M)} \leq C_M$ for all $x \geq 0$.
\end{proof}

We are now ready to state our quantitative Sturm bound for cuspidal Jacobi forms.
\begin{theorem}[Quantitative Sturm bound]\label{thm:QSBtake1} Fix a weight $\rho$, an open compact subset $X$ of $G(\A_f)$, and an open compact subgroup $U$ of $G(\A_f)$ satisfying $XU = X$.  Let $J_0 \subseteq J^\vee(\Q)$ be as in Lemma \ref{lem:WhitTsupport} and let $X_J$ be as in equation \eqref{eqn:MQ'decomp}.  Let $X_1 = X_J X$ and $Y = Q'(\A_f) X$.  There are positive constants $D$, $R=R_{\rho,X,U}$ and $B=B_{\rho,X,U}$ with the following property: Suppose $\eta \in SJ_{\rho,m}(Y,U)$ is a holomorphic cuspidal Jacobi form of weight $\rho$, index $m$, domain $Y$ and level $U$.  Let $\beta_{T}$ denote the normalized absolute Fourier coefficients of $\eta$; see Definition \ref{defn:betaTuT}.  If $\epsilon \geq 0$ and $\beta_T(x) \leq \epsilon$ for all $x \in X_1$ and $T \in J_0$ with $\det(T) \leq m \log(m)^{r-1} R_{\rho,X,U}$, then $||\eta||_Y \leq  \epsilon m^{D} e^{-2\pi r} B_{\rho,X,U}$.  In particular, $\beta_T(g) \leq \epsilon m^{D} B_{\rho,X,U}$ for all $T \in J^\vee$ and $g \in Y$.
\end{theorem}
\begin{proof} Suppose $y \in Y$ and $g \in G(\R)$.  Then
\begin{align*}
\eta(gy) &= \rho(J(g,i 1_J))^{-1} F_{y}(g) \\
&= \sum_{T \in J^\vee, (T,e) = m} \rho(J(g,i))^{-1} a_T(y) q^T\\
&= \sum_{T \in J^\vee, (T,e) = m} \rho(J(u_T g,i))^{-1} (\rho(J(u_T,i))a_T(y)) e^{2\pi i (1_J, u_T g (i 1_J))}.
\end{align*}

The $K_M$-invariant norm on $V_{\rho}$ induces a $K_M$-invariant matrix norm on $\End(V_{\rho})$, which we also denote by $|| \cdot ||$.  For ease of notation, set $Y_g = Im(g \cdot i 1_J)$ and
\[M_{T,\rho}(g) = || \rho(J(u_T g,i 1_J))^{-1} || e^{-\pi (T, Y_{g})}\]
where here we are using the matrix norm on $\End(V_{\rho})$.  The quantity $M_{T,\rho}(g)$ is a positive number independent of the choice of $u_T$.

It follows that
\[
||\eta(gy)|| \leq \sum_{T \in J^\vee, (T,e) = m} \beta_T(y) M_{T,\rho}(g) e^{-\pi (T,Y_g)}.\]
The function $M_{1_J,\rho}(g)$ is bounded for $g \in G(\R)$.  Since $M_{1_J,\rho}(u_T g) = M_{T,\rho}(g)$, we have some $C_\rho > 0$ so that $M_{T,\rho}(g) \leq C_\rho$ for all $T \in J(\R)^\vee$ positive-definite and all $g \in G(\R)$.  We obtain, for any $B > 0$,
\begin{align}\label{eqn:boundOnEtagy}
C_\rho^{-1}||\eta(gy)|| &\leq  \sum_{T \in J^\vee, (T,e) = m} \beta_T(y) e^{-\pi(T,Y_g)} \nonumber \\
&= \sum_{T \in J^\vee, (T,e) = m, \det(T) \leq m B} \beta_T(y) e^{-\pi(T,Y_g)} + \sum_{T \in J^\vee, (T,e) = m, \det(T) \geq m B} \beta_T(y) e^{-\pi(T,Y_g)}.
\end{align}

In the notation of Lemma \ref{lem:JacobiReduction}, fix some open compact subgroup $U_1$ of $G(\A_f)$ with $U_1 X_1 = X_1$, and suppose now that $g = n_1 m_1 tk$, with $n_1 \in N_Q^{U_1}$, $m_1 \in S_Q(\epsilon_J)$, $t \in \lambda(\R^\times_{>0})$ and $k \in K$.  Suppose also that $y \in X_1 = X_J X$.  For such a $g$ and $y$, we will bound the two sums in \eqref{eqn:boundOnEtagy} separately.

Recall the direct sum decomposition $J = \Q e_{11} \oplus V_0 \oplus J_1$ of equation \eqref{eqn:J_directSum3}.  This decomposition induces a direct sum decomposition of $J^\vee$.  If $T \in J^\vee$, we write $T = (m, x_0,T_0)$ for the pieces of $T$ in this decomposition of $J^\vee$ and likewise if $S \in J$ we do the same.

Suppose then $T = (m, x_0, T_0)$.  There is $x_1 \in V_0 \subseteq J$ so that
\[T_1:=T \cdot n_1 = (m, x_0 + m x_1, T_0 + x_0^* x_1 + x_1^* x_0 + m x_1^* x_1).\]
Set $Y_1 = m_1 \cdot 1_{r-1}$ where recall $1_{r-1} = \sum_{j=2}^{r}e_{jj}$.  We have $Y_1 > \epsilon_J 1_{r-1}$. It follows that
\begin{align*}
(T, Y_g) &= (T, (n_1 m_1 t) \cdot 1_J) = (T_1, (t^2, 0, Y_1)) \\
&= ((m, x_0 + m x_1, T_0 + x_0^* x_1 + x_1^* x_0 + m x_1^* x_1), (t^2,0,Y_1)) \\
&= t^2 m + (T_0 + x_0^* x_1 + x_1^* x_0 + m x_1^* x_1,Y_1) \\
&\geq t^2 m + \epsilon_J\left(\tr(T_0) + \frac{m}{2}(x_1,x_1) + (x_0,x_1)\right).
\end{align*}

Positive-definiteness of $T = (m, x_0, T_0)$ implies $T_0 - \frac{1}{m} x_0^* x_0 > 0$.  Taking trace gives $(x_0,x_0) <  2m (T_0, 1_{r-1})$, so
\[
(x_0, x_1) \geq - ||x_0|| \cdot ||x_1|| \geq - c_{G} m^{1/2} (T_0,1_{r-1})^{1/2}.
\]
Here $c_G > 0$ is a positive constant depending only on $G$ coming from reduction theory.

Thus
\[
(T,Y_g) \geq \epsilon_J\left(\tr(T_0) - c_G m^{1/2} \tr(T_0)^{1/2}\right) = \epsilon_J\cdot\tr(T_0)\left( 1 - c_G \left(\frac{m}{\tr(T_0)}\right)^{1/2}\right).
\]

Note that if $\tr(T_0) \geq 4 c_G^2 m$, then $c_G \left(\frac{m}{\tr(T_0)}\right)^{1/2} \leq 1/2$ so $(T,Y_g) \geq \epsilon_J \tr(T_0)/2$. We obtain
\[
\sum_{T \in J_0, (T,e) = m, T > 0} e^{-\pi (T,Y_g)} \leq \sum_{T \in J_0, \tr(T_0) < 4c_G^2m }{1} + \sum_{T \in J_0, \tr(T_0) \geq 4c_G^2 m}e^{-\frac{\pi}{2}\epsilon_J \tr(T_0)}.
\]
Applying Lemma \ref{lem:deltaSum}, it is now easy to see that the right-hand side is bounded by a polynomial in $m$, say it is $O(m^D)$. Indeed, since $(x_0,x_0) < 2m \tr(T_0)$, the second sum is bounded by
\[
\sum_{T_0 > 0} C' (m \tr(T_0))^{D'} e^{-\epsilon_J \pi\tr(T_0)/2}
\]
for some $C', D' > 0$. 

We now must consider the tail of the sum \eqref{eqn:boundOnEtagy}, i.e., 
\[
\sum_{T \in J_0, (T,e) = m, \det(T) \geq m B}e^{-\pi  (T,Y_g)}.
\]
Recall that $T_1 = T\cdot n_1 = (m, x_0+mx_1, T_2)$ with $T_2 = T_0 + x_0^* x_1 + x_1^* x_0 + m x_1^* x_1$.  Then $(T_1, Y_{m_1 t}) \geq \epsilon_J (T_2, 1_{r-1})$ and
\[
(T_2, 1_{r-1}) \geq \det(T_2)^{1/(r-1)} \geq (\det(T_1)/m)^{1/(r-1)} = (\det(T)/m)^{1/(r-1)}.
\]
The first inequality is AM-GM and the spectral theorem and the second is by Lemma \ref{lem:detIneq3part}.

Now
\[
2 (T,Y_g) \geq (T,Y_g) + \epsilon_J (\det(T)/m)^{1/(r-1)}.
\]
Thus
\[
\sum_{T \in J_0, (T,e) = m, \det(T) \geq mB }e^{-\pi (T,Y_g)} \leq e^{-2\epsilon_J B^{1/(r-1)}} \left(\sum_{T} e^{-\pi (T,Y_g)/2}\right).
\]

By Lemma \ref{lem:JacobiReduction}, we may choose $g,y$ so that $||\eta||_Y = ||\eta(gy)||$.  Suppose that $\beta_T(y) \leq \epsilon$ for all $T \in J_0$ with $\det(T) \leq m B$.  Putting the pieces together, we obtain
\begin{align*}
||\eta||_Y &= ||\eta(gy)|| \\
&\leq C_\rho \sum_{T \in J^\vee, (T,e) = m} \beta_T(x) e^{-\pi (T,Y_g)} \\
&\leq C_\rho \epsilon \left(\sum_{T \in J_0, (T,e)=m, T > 0} e^{-\pi (T,Y_g)}\right) + C_\rho e^{-2\epsilon_J B^{1/(r-1)}} ||\eta||_Y \left(\sum_{\det(T) \geq m B} e^{-\pi (T,Y_g)/2}\right) \\
&\leq C'm^{D}\epsilon +  C' m^{D} ||\eta||_Y e^{-2\epsilon_J B^{1/(r-1)}}.
\end{align*}

Rearranging gives
\[
||\eta||_Y \leq \frac{C' m^{D} \epsilon}{1-C' e^{D \log(m) - 2\epsilon_J B^{1/(r-1)}}}
\]
so long as the denominator is positive.  This proves the theorem.
\end{proof}

\begin{remark} Taking $\epsilon = 0$ in Theorem \ref{thm:QSBtake1} gives an upper bound on the possible vanishing orders of Jacobi forms, see e.g., \cite[section 4]{AIP} in the case of $\Sp_4$.  It seems likely that the best possible result on vanishing orders is as follows: There is a positive constant $R_{\rho,X,U}$ with the property that if $\eta \in SJ_{\rho,m}(Y,U)$ has Fourier coefficients $a_T$, and if $a_T(x) = 0$ for all $x \in X_1$ with $\det(T) \leq R_{\rho,X,U} \cdot m$, then $\eta = 0$.  Theorem \ref{thm:QSBtake1} weakens this presumably tight bound of $O(m)$ with the presence of a power of $\log(m)$, but then gains numerical control of the supremum $||\eta||_Y$ of $\eta$.
\end{remark}

\section{Automatic convergence}\label{sec:autConv}

We now prove the automatic convergence theorem for cuspidal formal modular forms.
\begin{theorem}\label{thm:autConvCusp} Fix a finite-dimensional complex representation $\rho: M_P(\C) \rightarrow \GL(V_{\rho})$ and let $U$ be a compact open subgroup of $G(\A_f)$.  Suppose $X \subseteq G(\A_f)$ satisfies $XU = X$, $P(\Q) X = G(\A_f) = Q(\Q)X$ and $Q'(\Q) X \subseteq N_P(\A_f) R'(\Q) X$.  Then the map $S_\rho(U) \rightarrow S_\rho^{f}(X,U)$ is a $\C$-linear isomorphism.
\end{theorem}
\begin{proof} We already know that the map is injective and $\C$-linear.  

Suppose $(a_T)_T \in S_{\rho}^{f}(X,U)$ is a formal cuspidal modular form of weight $\rho$, level $U$ and domain $X$.  By property \ref{item:propertyPsymms} of Definition \ref{def:formalMFs} and the assumption that $G(\A_f) = N_P(\A_f) M_P(\Q) X$, we can uniquely extend $a_T$ to a function $G(\A_f) \rightarrow V_{\rho}$ again satisfying the $P$-symmetries on all of $G(\A_f)$.

For $g_f \in G(\A_f)$, define
\[
\beta_T(g_f) = ||\rho(J(u_T,i 1_J)) a_T(g_f)||.
\]
The main step of the proof is to check that, for each fixed $g_f$, the $\beta_T(g_f)$ grow polynomially in $T$.  We separate out this part:

\begin{claim}\label{claim:InductiveConvergence} Notation and assumptions as above, the $\beta_T(g_f)$ grow polynomially in $T$ for each fixed $g_f \in G(\A_f)$.
\end{claim}
\begin{proof} We begin by reducing the set of $g_f$ that must be considered.  For this, note that if $n_f \in N_P(\A_f)$, then $\beta_T(n_f g_f) = \beta_T(g_f)$.  Also, if $\gamma \in M_P(\Q)$, then $\beta_T(\gamma_f g_f) = \beta_{T \cdot \gamma}(g_f)$. Since $G(\A_f) = N_P(\A_f) M_P(\Q) X$, it suffices to prove the polynomial growth for $y \in X$.

We have $X = \bigsqcup_{k=1}^{N} y_k U$.  Let $V= \cap_{k=1}^{N} y_k U y_k^{-1}$.  If $y \in X$ and $v \in V$, then $vy = yu$ for some $u \in U$.  Define $U' = M_P'(\A_f) \cap V$, and then $\Gamma_{P,U'} = M_P'(\Q) \cap U' = M_P'(\Q) \cap V$.  Let $F_{e,U'} \subseteq M_P'(\Q)$ and $B_{e,U_M} > 0$ be as in the statement of Proposition \ref{prop:redMP}.  

Set $X' = F_{e,U'}^{-1} X$.  We have $X' = \bigsqcup_{k=1}^{N'} z_k U$.  Let $V'= \cap_{k=1}^{N'} z_k U z_k^{-1}$ and $X_J' = X_J V' \supseteq X_J$.  Because $X_J$ is compact, $X_J'$ is a finite union of left $V'$ cosets, $X_J' = \bigsqcup_{j=1}^{M} r_j V'$.  Thus if $y \in X_J' X' = X_J' F_{e,U_M}^{-1} X$, then $y = r_j z_k u$ for some $u \in U$.

Since $G(\A_f) = N_P(\A_f) M_P(\Q) X$, we have
\[
r_j z_k = n_{j,k} \delta_{(j,k),f} x_{j,k}
\]
for some $n_{j,k} \in N_P(\A_f)$, $\delta_{(j,k)} \in M_P(\Q)$ and $x_{j,k} \in X$.  Thus
\[
\beta_T(y) = \beta_T(r_j z_k u) = \beta_T(n_{j,k} \delta_{(j,k),f} x_{j,k}u) = \beta_{T \cdot \delta_{(j,k)}}(x_{j,k}).
\]
Because there are only finitely many $\delta_{(j,k)}$, there is $L > 0$ so that $\det(T \cdot \delta_{(j,k)}) \leq L \det(T)$ for all $T \in J^\vee$.

Fix $\delta > 1$, sufficiently close to $1$, and $Q > 0$, $D_0 > 1$ and $E > 0$, to be determined below.  Suppose $\det(T)=D > 1$, $n \geq 1$, and $D_0^{\delta^{n-1}} \leq D < D_0^{\delta^n}$.  Define
\[f(D) = Q \cdot (D_0 \cdot D_0^{\delta} \cdots D_0^{\delta^{n-1}})^{E} = Q \cdot D_0^{E \cdot \frac{\delta^{n}-1}{\delta-1}}. \]
By inducting on $n$, we will prove that $\beta_T(x) \leq f(\det(T))$ for all $x \in X$.  As $f(D) < Q  \cdot D^{E \cdot \frac{\delta}{\delta-1}}$, this will prove the theorem.

There is a lattice $J_0 \subseteq J^\vee(\Q)$, depending on $X$, with $\beta_T(x) \neq 0$ and $x \in X$ implies $T \in J_0$.  By Proposition \ref{prop:redMP}, there is $\gamma \in \Gamma_{P,U'}$ and $t \in F_{e,U_M} \subseteq M_P'(\Q)$ so that if $T_1:=T \cdot (\gamma t)$, then $(T_1, e) \leq B_{e,U_M} \det(T)^{1/r}$.

Now if $x \in X$, then
\[
\beta_T(x) = \beta_{T_1}(t^{-1} \gamma^{-1} x) = \beta_{T_1}(t^{-1} x),
\]
as $\Gamma_{P,U'}$ was chosen so that $\gamma^{-1} x \in xU$ for $\gamma \in \Gamma_{P,U'}$ and $x \in X$.  To bound $\beta_{T_1}(t^{-1}x)$, we will apply the quantitative Sturm bound, Theorem \ref{thm:QSBtake1}.  Let $m_1= (T_1, e) \leq B_{e,U_M} \det(T)^{1/r}$.  Take any real number $\tau$ satisfying $0 < \tau < r-1$.  To apply Theorem \ref{thm:QSBtake1}, it suffices to have a bound on $\beta_S(X_J X')$ for all $S$ with 
\[\det(S) \leq m_1^{1+\tau} \leq B \det(T)^{(1+\tau)/r}.\]

Finally, suppose for the inductive step that $D_0^{\delta^{n}} \leq \det(T) < D_0^{\delta^{n+1}}$.  Choose $\delta > 1$ so that 
\[
\delta \cdot \frac{1+\tau}{r}< 1,
\]
which can be done so long as $1+\tau < r$ or $\tau < r-1$.  Now choose $D_0 > 1$ so that 
\[B L D_0^{\delta^{n+1} (1+\tau)/r} < D_0^{\delta^n}\]
for all non-negative integers $n$. Then if $\det(S) \leq B \det(T)^{(1+\tau)/r}$,
\[
\det(S \cdot \delta_{(j,k)}) \leq L \det(S) \leq B L \det(T)^{(1+\tau)/r} \leq BL D_0^{\delta^{n+1}(1+\tau)/r} \leq D_0^{\delta^n}.
\]
By the induction hypothesis, it follows that if $y \in X_J' X'$,
\[
\beta_S(y) = \beta_{S \cdot \delta_{(j,k)}}(x_{j,k}) \leq f(\det(S \cdot \delta_{(j,k)})) \leq f(D_0^{\delta^n}).
\]

By Theorem \ref{thm:QSBtake1}, we conclude that 
\[
\beta_T(x) = \beta_{T_1}(t^{-1} x) \leq B' f(D_0^{\delta^n}) m_1^D \leq B'' f(D_0^{\delta^n}) \det(T)^{D/r} \leq B'' f(D_0^{\delta^n}) D_0^{\delta^{n+1} D/r}.
\]
Thus if we choose $E$ so that $B'' D_0^{\delta^{n+1} D/r} < D_0^{\delta^{n+1} E}$ for all $n$, the induction completes.
\end{proof}

It follows from Claim \ref{claim:InductiveConvergence} that the sum
\[
\sum_{T \in J^\vee(\Q), T > 0}a_T(g_f) e^{2\pi i (T, g \cdot (i 1_J))}
\]
converges absolutely (the set of $T$ with $a_T(g_f) \neq 0$ contained within a lattice $J_{0,g_f}$) and one can define
\[
\varphi(g_fg):=\rho(J(g, i 1_J))^{-1}\sum_{T \in J^\vee(\Q), T > 0}a_T(g_f) e^{2\pi i (T, g \cdot (i 1_J))}.
\]
Then $\varphi: G(\A) \rightarrow V_\rho$ is a function with the following properties:
\begin{enumerate}
    \item For each $g_f \in G(\A_f)$, it is smooth and moderate growth as a function on $G(\R)$;
    \item it is $Z(\g)$-finite;
    \item it is right $U$-invariant.
\end{enumerate}
Moreover, it is left-invariant by $P(\Q)$.  To finish the proof, we must check that it is left-invariant by $Q'(\Q)$.

Define $\varphi_m: G(\A_f) \times G(\R) \rightarrow V_{\rho}$ as
\[
\varphi_m(g_f g) = \rho(J(g, i 1_J))^{-1}\sum_{T \in J^\vee(\Q), T > 0, (T,e) = m}a_T(g_f) e^{2\pi i (T, g \cdot (i 1_J))}.
\]
By the $Q$-symmetries, there is an $\eta_m \in SJ_{\rho,m}(Y,U)$ with $\eta_m = \varphi_m$ upon restriction to $X \times G(\R)$.  We have
\[
\eta_m(y g) = \rho(J(g,i1_J))^{-1}\sum_{T \in J^\vee(\Q), T > 0, (T,e) = m} b_{T,m}(y) q^{T}
\]
for some Fourier coefficients $b_{T,m}: Y \rightarrow V_{\rho}$.  If $v \in J(\A_f)$, $\mu \in M_P(\Q) \cap Q'(\Q)$ and $x \in X$, the Fourier coefficients $b_{T,m}$ satisfy
\[
b_{T,m}(n(v) \mu_f x) = \psi((T,v)) \rho(J(\mu_\infty, i 1_J)) b_{T,m}(x).
\]
The Fourier coefficients $a_T$ satisfy the same relations, so $a_T = b_{T,m}$ on $N_P(\A_f) R'(\Q) X$.  But we assume $Q'(\Q) X \subseteq N_P(\A_f) R'(\Q) X$.  Thus if $\delta \in Q'(\Q)$, $x\in X$ and $g \in G(\R)$ then 
\begin{equation}\label{eqn:varphimQ'inv}
\varphi_m(\delta x g)= \varphi_m(\delta_f x \delta_\infty g) = \eta_m(\delta_f x \delta_\infty g) = \eta_m(x g) = \varphi_m(x g).
\end{equation}

Let $\nu: Q\rightarrow \GL_1$ denote the action on $n_L(e) \in \g_J$, so that $Ad(q) n_L(e) = \nu(q) n_L(e)$.  Suppose $t \in Q(\Q) \cap M_P(\Q)$.  Then 
\[\varphi_m(t g_f g) = \varphi_{\nu(t) m}(g_f g)\]
for all $g_f \in G(\A_f)$ and $g \in G(\R)$.  Indeed,
\begin{align*}
\varphi_m(t g_f g) &= \rho(J(t_\infty g,i1_J))^{-1}\sum_{T \in J^\vee(\Q), (T,e) = m} a_T(t_f g_f) e^{2\pi i (T, t_\infty g \cdot i 1_J)} \\
&= \rho( J(t_\infty g, i 1_J))^{-1} \sum_{T \in J^\vee(\Q), (T,e) = m} \rho(J(t_\infty, i 1_J)) a_{T \cdot t}(g_f) e^{2\pi i (T \cdot t_\infty, g \cdot (i1_J))} \\
&=\varphi_{\nu(t) m}( g_f g).
\end{align*}

Now suppose $g_f \in G(\A_f) = Q(\Q) X$, $g \in G(\R)$ and $\delta \in Q'(\Q)$.  We check that $\varphi_m(\delta g_f g) = \varphi_m(g_f g)$.  Write $g_f = \gamma_f x$ with $\gamma \in Q(\Q)$.  There is $t \in Q(\Q) \cap M_P(\Q)$ with $\nu(t) = \nu(\gamma)$.  We have $\delta \gamma = tv$ for some $v \in Q'(\Q)$.  Thus
\begin{align*}
\varphi_m( \delta g_f g) &= \varphi_m (\delta_f \gamma_f x \delta_\infty g) \\
&=\varphi_m(t_f v_f x \delta_\infty g) \\
&=\varphi_{\nu(t) m}(v_f x t_\infty^{-1}  \delta_\infty g)\\
&=\varphi_{\nu(t) m}(x v_\infty^{-1} t_\infty^{-1} \delta_\infty g).
\end{align*}
For the last equality we use \eqref{eqn:varphimQ'inv}.  But $v^{-1} t^{-1} \delta = \gamma^{-1}$ so
\begin{align*}
\varphi_m(\delta g_f g) &= \varphi_{\nu(t) m}(x \gamma_\infty^{-1} g) \\
&=\varphi_{\nu(t)m}( t_f^{-1} \gamma_f x t_\infty^{-1} g).
\end{align*}
In this last line we use that $t^{-1} \gamma \in Q'(\Q)$.  But now
\[
\varphi_{\nu(t)m}(t_f^{-1} \gamma_f x t_\infty^{-1} g) = \varphi_{m}(\gamma_f x g) = \varphi_{m}(g_f g).
\]

This proves that $\varphi_m$ is left-invariant by $Q'(\Q)$ on all of $G(\A_f) \times G(\R)$.  It follows that $\varphi$ on $G(\A)$ is left-invariant by $Q'(\Q)$.  Since $G(\Q)$ is generated by $P(\Q)$ and $Q'(\Q)$, $\varphi$ is left-invariant by $G(\Q)$.  This proves the theorem.
\end{proof}

To derive the automatic convergence theorem for $M_\rho^{f}(X,U)$, we use the Koecher principle and a result of Baily-Borel.
\begin{proof}[Proof of Theorem \ref{thm:autConvIntro}] Let $(a_T)_T \in M_\rho^{f}(X,U)$.  By Baily-Borel \cite{bailyBorel}, there is $\ell >>0$ and a compact open subgroup $U' \subseteq U$ so that for every $P \in G(\A)$, there is a classical weight $\ell$ cusp form $f_P$ of level $U'$ with $f_P(P)\neq 0$.  Multiplying the Fourier expansion of $f_P$ with the $(a_T)_T$, we obtain a new sequence $(b_{S,P})_S$.  It is easy to see that $(b_{S,P})_S \in S_{\rho(\ell)}^{f}(X,U').$  By Theorem \ref{thm:autConvCusp}, there is an $\varphi_{\rho(\ell),P} \in S_{\rho(\ell)}(U')$ with classical Fourier coefficients equal to the $b_{S,P}$.  Here $\rho(\ell)$ is the new representation of $M_P(\C)$ that comes from tensoring $\rho$ with the $1$-dimensional representation determined by $\ell$.

Let $\eta_P \in S_{\ell}(U')$ correspond to $f_P$ and let $f_{\rho(\ell),P}$ be the classical holomorphic cusp form corresponding to $\varphi_{\rho(\ell),P}$, i.e., $f_{\rho(\ell),P}(h_f h) = \rho(\ell)(J(h,i 1_J)) \varphi_{\rho(\ell),P}(h_f h)$.  Define $\varphi: G(\A) \rightarrow V_{\rho}$ as $\varphi(P) = \eta_P^{-1}(P) \varphi_{\rho(\ell)}(P)$.  It is immediately seen that $\varphi(P)$ is independent of the choice of $f_P$ with $f_P(P) \neq 0$, and thus $\varphi(P)$ is smooth.  It is left-invariant by $G(\Q)$ because $\eta_P$ and $\varphi_{\rho(\ell)}$ are, and likewise is right-invariant by $U'$.  It satisfies that $\rho(J(h,i 1_J)) \varphi(h_f h)$ defines a holomorphic function on $\H_J$.  Thus by the Koecher principle, $\varphi$ is of moderate growth, so $\varphi \in M_{\rho}(U')$. 

For every $P$, we have $\eta_P \varphi = \varphi_{\rho(\ell),P}$.  Looking at Fourier coefficients, we obtain that the Fourier coefficients of $\varphi$ are the $a_T$.  Since they are right $U$-invariant, in fact we have $\varphi \in M_{\rho}(U)$.  This completes the proof.
\end{proof}

\bibliography{autConvHMF_bib}

	\bibliographystyle{amsalpha}
\end{document}